\documentclass[a4paper,12pt]{amsart}

\usepackage[latin1]{inputenc}
\usepackage[T1]{fontenc}
\usepackage[english]{babel}
\usepackage{todonotes}

\usepackage[hyphens]{url}
\usepackage{hyperref}

\usepackage{amsmath,amssymb,amscd,amsthm,amsfonts,amstext,amsbsy,mathrsfs,upgreek,mathtools,stmaryrd,enumitem,bbm,ifsym} 

\usepackage{xspace}
\usepackage[all]{xy}
\usepackage{graphicx}
\usepackage{url}
\usepackage{latexsym}

\makeatletter
\newcommand*{\rom}[1]{\expandafter\@slowromancap\romannumeral {\sharp}1@}
\makeatother

\theoremstyle{definition}

\newcommand{\PP}{\mathbb{P}}

\DeclareMathAlphabet{\mymathbb}{U}{BOONDOX-ds}{m}{n}

\newcommand{\megj}[1]{}

\newtheorem{fact}{fact}

\newtheorem{thm}[fact]{Theorem}
\newtheorem{lemma}[fact]{Lemma}
\newtheorem{prop}[fact]{Proposition}
\newtheorem{corollary}[fact]{Corollary}
\newtheorem{defini}[fact]{Definition}
\newtheorem{remark}[fact]{Remark}

\newtheorem{question}[fact]{Question}

\title{Canonical Truth}

\author{Merlin Carl}
\address{Institut f\"ur mathematische, naturwissenschaftliche und technische Bildung, Abteilung f\"ur Mathematik und ihre Didaktik, Europa-Uni-versit\"at Flensburg, Geb\"aude RIGA 1, Auf dem Campus 1b, 24943 Flensburg}

\email{merlin.carl@uni-flensburg.de}

\author{Philipp Schlicht} 

\address{Philipp Schlicht, School of Mathematics, University of Bristol, Fry Building, Woodland Road, Bristol, BS8 1UG, UK} 

\email{philipp.schlicht@bristol.ac.uk} 

\keywords{Set theory, Mathematical realism, Philosophy of mathematics} 

\begin{document}

\maketitle

\begin{abstract}
We introduce and study some variants of a notion of canonical set theoretical truth. By this, we mean truth in a transitive proper class model $M$ of ZFC that is uniquely characterized by some $\in$-formula. We show that there are interesting statements that hold in all such models, but do not follow from ZFC, such as the ground model axiom and the nonexistence of measurable cardinals. 

We also study a related concept in which we only require M to be fixed up to elementary equivalence. We show that this theory-canonicity also goes beyond provability in ZFC, but it does not rule out measurable cardinals and it does not fix the size of the continuum.
\end{abstract}

\section{Introduction}

It is an old logical dream to devise an effectively describable axiomatic system for mathematics that uniquely describes `mathematical reality'; in modern logical language, this should mean
at least that it uniquely fixes a model.
It is well-known that this dream is unattainable in first-order logic: By the L\"owenheim-Skolem theorem, we get models of all infinite cardinalities once there is
one infinite model; and by G\"odel's incompleteness theorem, if the theory is strong enough to express elementary arithmetic, it will have different models that are not even elementary equivalent.

Focusing on ZFC set theory, one of the main foundational frameworks for mathematics, these two effects can in a certain sense be cancelled out by asking not for arbitrary models, but for
transitive models that are proper class-sized, i.e. contain all ordinals. When we restrict the allowed models in this way, there are extensions of ZFC that uniquely fix a model.
The most prominent example is $V=L$: It is well-known (and provable in ZFC) that ZFC+$V=L$ has exactly one transitive class-model, provided that ZFC is consistent.

This form of canonicity gives the axiom of constructibility a certain attractiveness: It seems to describe, up to the unavoidable weakness of first-order logic, a unique `mathematical reality'. However, 
it is usually seen as too restrictive since many objects of set-theoretical interest are ruled out under this assumption.

But $V=L$ is by far not the only theory that uniquely fixes a transitive class model: Other examples include $V=L[0^{\sharp}]$ and $V=L[x]$, where $x$ is an absolute $\Pi^{1}_{2}$-singleton (see below).
The `true mathematical reality' that the adherents of the logical dream mentioned in the beginning believe in would have to be one of those `canonical' models. Hence, whatever holds in all of these `canonical' models
will have to be accepted as true by someone who believes in a uniquely describable mathematical reality. We call such statements `canonically necessary'.
If there are no such statements that go beyond what is derivable from ZFC, then this kind of mathematical realism would be mathematically neutral:
the belief in a uniquely describable mathematical reality would merely be a way of interpreting set theory, without influencing it. On the other hand, if there are statements that hold in all canonical models without
following from ZFC, this realistic mindset would be mathematically informative.

In this paper, we investigate statements that hold in all `canonical' models of ZFC, i.e. in all transitive class models that are uniquely fixed by some extension of ZFC by finitely many extra statements.\footnote{After most of the work in this paper was done, we noticed that in \cite{HFriedman1} and \cite{HFriedman2}, H. Friedman defined and investigated a similar concept for countable set-sized models of a fixed height $\alpha$ under the name `$\beta$-categoricity'. However, there is otherwise no overlap in the settings, the questions considered and the results; in particular, in Friedman's setting, one of the main results is that $L$ is the unique `categorial' model (\cite{HFriedman1}, p. 543), while in our setting, there are infinitely many under sufficient large cardinal assumptions.} 
It turns out that the realistic mindset is indeed mathematically informative: Examples of canonically necessary statements that do not follow from ZFC are the ground model axiom of (\cite{R}) (Theorem \ref{gmacn})
and the non-existence of measurable cardinals (Theorem \ref{nomeasurables}).

This approach generalizes in a natural way to the concept of `canonical consequence': Namely, a sentence $\psi$ is a `canonical consequence' of a theory $T$ if and only if, for all $\phi$ such that $T+\phi$ has (provably in ZFC\footnote{It would also be natural to replace ZFC by $T$ here as well; but for the time being, we keep ZFC as our base theory.}) exactly one transitive class model $M$, we have $M\models\psi$.

One can then ask for the canonicity of the ZFC axioms themselves: Are there proper subsystems $T$ of ZFC that canonically imply the ZFC axioms? We give some preliminary results in this respect.\footnote{A related question is which subsystems $T$ of ZFC have the property that ZFC holds in all inner models of $T$. One result in this direction, due to Philip Welch, will be given below.}

A natural weakening of the concepts of canonical truth and consequence would be to merely demand that the theory of the transitive class model $M$, rather than $M$ itself, is uniquely determined by an $\in$-sentence. Let us say that an $\in$-theory is canonical if and only if there is an $\in$-sentence $\phi$ such that all transitive class models of ZFC+$\phi$ are elementary equivalent. Then we can say that a sentence $\psi$ is `theory-canonical' if and only if it is contained in all canonical theories. If we replace ZFC by some $\in$-theory $T$ in this definition, we obtain the notion of a $T$-canonical theory. We can then say that a sentence $\psi$ is a `theory-canonical consequence of $T$' if and only if  it is contained in all $T$-canonical theories. This `theory canonicity' turns out to be strictly weaker than canonicity (for example, it does not exclude measurable cardinals), but is still informative: There are theory-canonical consequences of ZFC that are not first-order consequences of ZFC.

We conclude with various open questions; in particular, we do not know whether 
the continuum hypothesis is canonically necessary (we conjecture that it is not) or whether there are canonical models of ZF+$\neg$AC (i.e. whether the axiom of choice is canonically necessary over ZF).

\section{Basic Definitions}

We start by giving a formal counterpart to the intuitive idea that a theory $T$ `uniquely fixes a transitive class model' and `uniquely fixes a transitive class model up to elementary equivalence'. This is not straightforward, as quantifying over proper classes
is not possible in ZFC. This might be solvable by instead working in NBG, but we prefer to stick to ZFC for the moment, partly because the methods we intend to use (forcing, class forcing and inner models)
are commonly developed for ZFC models. Thus, a proper class model of ZFC will always be an inner model of $V$. Of course, this will immediately trivialize our analysis
when one assumes $V=L$, so that $L$ is the only transitive class model. To get a sufficient supply of inner models, we will hence assume sufficient large cardinals in our metatheory.

Still, we need to deal with our inability, due to the lack of a truth predicate, to quantify over all inner models. This will be solved by formulating the uniqueness not as a single statement,
but as a scheme. This leaves us with the problem of expressing that the class defined by a formula $\phi$ is a model of ZFC. Again, this is not trivial, since ZFC is not finitely axiomatizable.
Fortunately, for the case we are interested in, there is a workaround:

\begin{lemma}{\label{ZFmodelchar}}
[See \cite{Je}, Theorem $13.9$.]
 A transitive class $C$ is a model of ZF if and only if $C$ is closed under G\"odel operations and almost universal (i.e. for every subset $x\subseteq C$, there is $y\in C$ with $x\subseteq y$).
\end{lemma}

We fix a natural enumeration $(\psi_{i}:i\in\omega)$ of the $\in$-formulas in order type $\omega$.

\subsection{Canonical Implication}

\begin{defini}
Let $\phi$ be an $\in$-formula, $i,j\in\omega$.
Let $\text{IM}_{i}^{\text{ZFC}}(\phi
,y)$ (`inner model') abbreviate the statement `$M_{\psi_{i},y}:=\{x:\psi_{i}(x,y)\}$ is transitive, almost universal, closed under G\"odel operations,
contains all ordinals and satisfies AC and $\phi$'. More generally, when $T$ is an $\in$-theory, we let $\text{IM}_{i}^{\text{T}}(\phi
,y)$ denote the claim that $M_{\psi_{i},y}$ is a transitive class model of $T$.

The uniqueness statement $U_{ij}^{\phi
,T}$ is the following $\in$-formula: $$\forall{y,y^{\prime}}[(\text{IM}_{i}^{T}(\phi
,y) \wedge \text{IM}_{j}^{T}(\phi
,y^{\prime}))\rightarrow \forall{x}(\psi_{i}(x,y)\leftrightarrow\psi_{j}(x,y^{\prime}))].$$

Now, $\phi$ 
 is a uniqueness statement over a theory $T$ if and only if all elements of $U_{\phi}^{T}:=\{U_{ij}^{\phi
,T}:i,j\in\omega\}$ are provable in 
T\footnote{Alternatively, we could also demand that all elements of $U_{\phi}^{T}$ hold in $V$. We will take up this idea below as $C^{1}$-canonicity.}. When $T$ is ZFC, we will usually drop `over $T$'.

Moreover, for $T$ an extension of KP, $\phi$ 
is a $T$-canonical statement if and only if there is some $i\in\omega$ such that $\exists{y}\text{IM}_{i}(\phi,y)$ and $\phi$
 is a uniqueness statement. When $T$ is ZFC, we simply call $\phi$ 
canonical.

\begin{remark}
Note that we do \textit{not} require $\exists{y}\text{IM}_{i}(\phi,y)$ to be provable in $T$; we only want it to be true (in $V$). These existence statements will usually be derived from stronger meta-theories, such as ZFC with large cardinals.
\end{remark}

\end{defini}

\begin{remark} Typically, $T$ will just be ZFC. Below, we will also consider cases where $T$ is much weaker; however, some base theory is necessary to exclude unwanted cases such as `$\in$ is a total ordering' (which has $\text{On}$ as its only transitive proper class model) from our consideration. 
\end{remark}

The paradigmatical example for a uniqueness statement is $V=L$, which is a $\Pi_2$-statement. This is indeed the minimal complexity for a uniqueness statement:

\begin{prop}
No $\Sigma_2$-sentence is a uniqueness statement.
\end{prop}
\begin{proof}
Suppose otherwise, let $\psi\equiv\exists{x}\forall{y}\phi(x,y)$ be a uniqueness statement, where $\phi$ is $\Delta_0$, and let $M$ be the unique transitive class model of ZFC+$\psi$. Pick $a\in M$ such that $M\models\forall{y}\phi(a,y)$, and let $\kappa\in\text{Card}$ be large enough so that $a\in V_{\kappa}^{M}$. As a cardinal, $\kappa$ is $\Sigma_1$-reflecting, so $V_{\kappa}^{M}\models\forall{y}\phi(a,y)$. 
Let $\mathbb{P}_{\kappa^{+}}$ be the forcing for adding a new subset of $\kappa^+$ described in (\cite{Kunen}, section $6$) (i.e., the set of partial functions $f:\kappa^{+}\rightarrow\{0,1\}$ of cardinality $<\kappa^{+}$), and let $M[G]$ be a generic extension for $\mathbb{P}_{\kappa}$. Since $\mathbb{P}_{\kappa^{+}}$ is $\kappa$-closed, we have $V^{M[G]}_{\kappa}=V^{M}_{\kappa}$; thus $V^{M[G]}_{\kappa}\models\forall{y}\phi(a,y)$, and since $\kappa$ is also a cardinal in $M[G]$, we have that $M[G]\models\forall{y}\phi(y,a)$ by $\Sigma_1$-reflection. Thus $M[G]\models$ZFC+$\psi$, contradicting the assumption that $\psi$ is a uniqueness statement.
\end{proof}

\begin{defini}
A statement $\phi$ is canonically necessary (c.n.) if and only if $\phi$ holds in all canonical models of ZFC.

A statement $\phi$ is canonically possible (c.p.) if and only if there is a canonical model $M\models\phi$ of ZFC, i.e., if and only if its negation is not canonically necessary. 

\end{defini}

\begin{defini}
Let $T$ be an extension of KP.

If $M$ is a transitive class model of T, then $M$ is $T$-canonical if 
and only if there is a $T$-canonical statement $\phi$ such that $M\models\phi$.

If $A$ is any $\in$-theory and $\phi$ is an $\in$-statement, 
then $\phi$ canonically follows from $A$ if and only if $\phi$ holds in all canonical models in which $A$ holds. In this case, we write $A\models_{c}\phi$.
\end{defini}

The preceding notions can be generalized by replacing the single tc statement $\phi$ with elements from a class $\mathcal{T}$ of theories. We would then, e.g., say that $T\models_{c}^{\mathcal{T}}\phi$ if and only if $\phi$ holds in every transitive proper class model of $T$ that is fixed by some element of $\mathcal{T}$. 
Particularly interesting cases might be the set of recursive theories, of countable theories, of $\Sigma_n$-axiomatizable theories, or even the class of all $\in$-theories with ordinal parameters.
However, in this work, we will focus on the singleton case and only briefly mention when our results easily generalize to other variants.

For a class $\mathcal{A}$ of $\in$-theories with ordinal parameters, we say that a formula $\phi$ is 
$\mathcal{A}$-canonical, written ZFC$\models_{c}^{\mathcal{C}}\phi$, if and only if $\phi$ holds in all transitive class models $M$ of ZFC such that, for some $T\in\mathcal{A}$, $M$ is the unique transitive class model of ZFC+T.\footnote{Note that we do not require that the uniqueness of the respective models is provable in ZFC, but only that it is true (in $V$). This relaxation seems unavoidable, as no analogue of the schematic approach taken for single statements is available for arbitrary theories. }

Let $\mathcal{C}$ denote the class of 
$\in$-theories with ordinal parameters. 
For a cardinal $\kappa$, $\mathcal{C}^{\kappa}$ denotes the subclass of $\mathcal{C}$ consisting of the elements with cardinality $\kappa$.\footnote{Note that, by the last footnote, $\mathcal{C}^{1}$-canonical necessity is not the same as canonical necessity.} 
Finally, let $\mathcal{C}^{L}$ be the class of constructible countable $\in$-theories.

\subsubsection{Weak Canonicity}

In our definition above, we required that ZFC must be capable of proving
the uniqueness of a model of ZFC+$\phi$. A somewhat reasonable weaker requirement would be that ZFC+$\phi$ proves this.

\begin{defini}
 $\phi$ is weakly canonical if and only if ZFC+$\phi$ proves the uniqueness statements in the definition of canonicity. If $M$ is a model of ZFC+$\phi$
for some weakly canonical $\phi$, then $M$ is weakly canonical. If $\psi$ holds in all weakly canonical models, then $\psi$ is weakly canonically necessary (weakly c.n.).
\end{defini}

\subsection{Theory-canonical implication}

We now give a formal counterpart to the claim that a statement $\phi$ fixes a transitive class model of a theory $T$ up to elementary equivalence. 
Intuitively, a sentence $\phi$ is `theory-canonical' (tc) if and only if it is provable in ZFC that any two inner models $M$ and $N$ of ZFC+$\phi$ are elementary equivalent.
Formally, we again need to express this as a scheme: Whenever $\phi_{0}$ and $\phi_{1}$ define transitive proper classes $M_{0}$ and $M_{1}$ in which $\phi$ holds, 
we have that ZFC proves $\psi^{M_0}\leftrightarrow\psi^{M_1}$.

\begin{defini}{\label{theory canonical}}

For $i,j,k\in\omega$, $\phi$ an $\in$-formula 
and $T$ an $\in$-theory, the theory-uniqueness statement TU$^{\phi,
T}_{ijk}$ is the following $\in$-formula:
$$\forall{y,y^{\prime}}[(\text{IM}_{i}^{T}(\phi
,y)\wedge\text{IM}_{j}^{T}(\phi
,y^{\prime}))\rightarrow(\psi_{k}^{M_{\psi_{i},y}}\leftrightarrow\psi_{k}^{M_{\psi_{j},y^{\prime}}})].$$

An $\in$-sentence $\phi$ 
is a theory-uniqueness statement for $T$ if and only if ZFC proves all elements of $\text{TU}_{\phi}^{T}:=\{\text{TU}_{ijk}^{\phi
}:i,j,k\in\omega\}$.

We say that a theory-uniqueness sentence $\phi$ for $T$ 
is theory-canonical (tc) over a theory $T$ if and only if $\exists{y}\text{IM}_{i}\phi(\phi
,y)$ is true (in $V$) for some $i\in\omega$. 

For theories $T$ and $S$, we say that $S$ is a canonical extension of $T$ if and only if there is a theory-canonical sentence $\phi$ over $T$ such that $S$ is the (unique) theory of the transitive class models of $T+\phi$. When $T$ is ZFC, we call $S$ a canonical theory. 

 A transitive class model $M$ of $T$ is tc over $T$ if and only if some $\phi$ holds in $M$ which is tc over $T$.
 
 A sentence $\psi$ is `theory-canonically necessary' (tcn) if and only if $\psi$ 
 belongs to all canonical extensions of ZFC. 
 Likewise, $\psi$ is `theory-canonically possible' (tcp) if and only if it belongs to some canonical extension of ZFC. 
\end{defini}

We can then extend this as above to obtain a notion of `theory-canonical implication':

\begin{defini}
If $T$ is an $\in$-theory and $\phi$ is an $\in$-sentence, then we say that $\phi$ is a `theory-canonical consequence (tc-consequence) of $T$', or that $T$ `tc-implies $\phi$' if and only if 
we have $M\models\phi$ for every transitive class model $M$ that is tc over $T$. In this case, we write $T\models_{\text{tc}}\phi$.
\end{defini}

As above, the notions of theory-canonicity and tc-implication can be generalized to theories, rather than single statements.

\subsection{General properties of $\models_{c}$ and $\models_{\text{tc}}$}

In this section, we consider some general logical and proof-theoretic properties of $\models_{c}$ and $\models_{\text{tc}}$.

\begin{defini}
For a theory $T$, we let $\mathcal{C}(T):=\{\phi:T\models_{c}\phi\}$ and $\mathcal{TC}(T):=\{\phi:T\models_{\text{tc}}\phi\}$ be the sets of canonical and theory-canonical consequences of $T$, respectively. 
\end{defini}

Note that, for any $T$, we have $\mathcal{C}(T)\subseteq\mathcal{CT}(T)$. It is easy to see that both $\models_{c}$ and $\models_{\text{tc}}$ are closure operators on the set of $\in$-theories.\footnote{Note a slight pathology for theories $T$ that do not have (theory-)canonical models: For such theories, the definition of (theory-)canonical implication universally quantifies over the empty set and is thus vacuously fulfilled for every statement $\phi$, so that, in this case, we have that $\mathcal{C}(T)$ and $\mathcal{TC}(T)$ both coincide with the set of all $\in$-sentences.}

Clearly, for any theory $T$, $\mathcal{C}(T)$ and $\mathcal{TC}(T)$ are closed under first-order inferences. There are, however, additional inference rules with this property, such as the $\omega$-rule (for which, however, the restriction to transitive models suffices). 
We do not know whether there are inference rules that are sound for $\models_{c}$ or $\models_{\text{tc}}$ but not already for transitive models.

We also note that, again due to the restriction to transitive models, compactness fails for both notions. Since in the setting of (theory-)canonicity, we want to keep our basic theory even if it is infinite, we define the relevant compactness property as follows:

\begin{defini}
Let $T$ be an $\in$-theory. We say that $\models_{\text{c}}$ has the compactness property over $T$ if and only if, for every set $S$ of $\in$-sentences $\phi$, $T\cup S\models_{c}\phi$ implies that there is a finite subset $S^{\prime}\subseteq S$ such that $T\cup S^{\prime}\models_{c}\phi$. The definition for $\models_{\text{tc}}$ is analogous.
\end{defini}

\begin{prop}{\label{not compact}}
Neither $\models_{c}$ nor $\models_{\text{tc}}$ have the compactness property. That is, there are a theory $T$ and an $\in$-sentence $\phi$ such that $T\models_{\text{tc}}\phi$ (and thus $T\models_{\text{c}}\phi$), but there is no finite subtheory of $T^{\prime}$ of $T$ such that $T^{\prime}\models_{\text{c}}\phi$ (and thus, neither $T^{\prime}\models_{c}\phi$). 
\end{prop}
\begin{proof}
Define $0^{\sharp,i}$ as the $i$-th iterate of the sharp operator applied to $0$. 
For $i\in\omega$, let $\phi_{i}$ be the sentence `$0^{\sharp,i}$ exists',  let $S:=\{\phi_{i}:i\in\omega\}$, let $T$ be ZFC$\cup $S, and let $\psi$ be the sentence `For every $i\in\omega$, $0^{\sharp,i}$ exists'. Then $T\models_{\text{tc}}\psi$ (and hence $T\models_{\text{c}}\psi$). 

However, if $S^{\prime}$ is a finite subset of $S$, then there is a $k\in\omega$ that is maximal with the property that $\phi_{k}\in S^{\prime}$. But now, $V=L[0^{\sharp,k}]$ is a uniqueness statement over ZFC and in the unique transitive class model of $V=L[0^{\sharp}]$, $\psi$ clearly fails.
Hence $T^{\prime}\not\models_{\text{c}}\psi$ (and thus also $T^{\prime}\not\models_{\text{c}}\psi$). 

Thus, $T$ and $\psi$ are as desired.
\end{proof}

Since compactness fails, there cannot be sets of finite inference rules (i.e., rules with finitely many premises) that are both sound and complete for $\models_{c}$ or $\models_{\text{tc}}$.

\begin{prop}{\label{not compact 2}}
There is an infinite $\in$-theory $T$ such that, for any finite subset $S\subseteq T$, ZFC+S has a canonical model, but ZFC+T does not.
\end{prop}
\begin{proof}
Let $T$ and $\psi$ be as in the proof of Proposition \ref{not compact}, then $T+\neg\psi$ is as desired.
\end{proof}

By taking $\psi$ to be $\bot$, Proposition \ref{not compact} follows from Proposition \ref{not compact 2}. However, we consider the given presentation to be more illustrative.

\section{Examples of Canonical Truth}

Obvious examples for uniqueness statements are $V=L$ or $V=L[0^{\sharp}]$ with corresponding canonical models $L$ and $L[0^{\sharp}]$. These actually give rise to a larger class of examples:

\begin{defini}
 A real number $x$ is a relative $\Pi^{1}_{2}$-singleton if and only if there is a $\Pi^{1}_{2}$-statement $\phi$ such that $x$ is the unique element $y$ of $L[x]$ with $L[x]\models\phi(y)$.

A real number $x$ is an absolute $\Pi^{1}_{2}$-singleton\footnote{For the notion of an absolute $\Pi^1_2$-singleton, cf., e.g., David \cite{David}.} if and only if there is a $\Pi^{1}_{2}$-statement $\phi$ such that $x$ is the unique element $y$ of $V$ with $\phi(y)$.
\end{defini}

\begin{corollary}
An absolute $\Pi^{1}_{2}$-singleton $x$ is the unique element satisfying its defining $\Pi^{1}_{2}$-formula $\phi$ in each transitive inner model that contains $x$, while
all other models will not contain such a witness.
\end{corollary}
\begin{proof}
 By Shoenfield absoluteness, if $M$ is a transitive class model of $ZF^{-}$, $x\in M$ and $M\models\phi(x)$, then $V\models\phi(x)$. Hence, if some transitive inner model had two distinct elements
satisfying $\phi$, the same would hold for $V$, contradicting uniqueness. Similarly, if $M$ was some transitive inner model with $x\notin M$ but $M\models\phi(y)$ for some $y\in M$,
then $V\models\phi(x)\wedge\phi(y)\wedge x\neq y$, again contradicting the uniqueness.
\end{proof}

The existence of $0^{\sharp}$ has consistency strength. However, no such assumption is needed to obtain canonical models beyond the constructible universe:

\begin{prop}
It is consistent relative to ZFC that there are canonical models besides $L$.
\end{prop}
\begin{proof} (Sketch)
 Force a $\Pi^{1}_{2}$-singleton over $L$ as described in chapter $6$ of \cite{Fr}. The generic extension satisfies that there is a real number $r$
satisfying the $\Pi^{1}_{2}$-statement $\psi$ (which is unique) and $V=L[r]$ and is unique with this property.
\end{proof}

\begin{lemma}
 There is a subtheory $T$ of ZFC+`$0^{\sharp}$ exists' that is consistent with $V=L$ such that there are canonical models $M$ of $T$ with $L\subsetneq M\subseteq L[0^{\sharp}]$. 
\end{lemma}
\begin{proof}
In \cite{Fr}, Theorem 6.23, it is shown that there is a $\Pi^{1}_{2}$-singleton $x$ and a theory $T$ as in the statement of the lemma such that $x\notin L$, $0^{\sharp}\notin L[x]$ and $L[x]$ is canonical with respect to $T$.
\end{proof}

According to the discussion before Theorem 6.23 in \cite{Fr}, it is open whether the $x$ in the preceding proof can be proved to be unique in ZFC rather than ZFC+$0^{\sharp}$ exists.

\begin{question}
Is there a canonical model $M$ with $L\subsetneq M\subsetneq L[0^{\sharp}]$?
\end{question}

Our first observation concerning canonical implication is that it does not coincide with first-order provability, i.e., there are canonically necessary statements that are not provable in ZFC:

\begin{lemma}
 There is some $\in$-formula $\phi$ such that $\phi$ does not hold in all transitive class models of ZFC, but $\phi$ is canonically necessary.
\end{lemma}
\begin{proof}
 Let $\phi$ be the statement: `It is not the case that there is a Cohen-generic filter $G$ over $L$ such that $V=L[G]$'. (Thus, intuitively, $\phi$ says: `I am not a Cohen-extension of $L$').
This is an $\in$-statement. Clearly, $\phi$ is false in a Cohen-extension $L[G]$ of $L$.

On the other hand, let $M$ be canonical and assume that $M\models\phi$. Let $\psi$ be a uniqueness statement for $M$. Then there is some $G$ Cohen-generic over $L$ with $M=L[G]$. Moreover,
as $M\models\psi$, there is some condition $p$ such that $p\Vdash\psi$. Let $G^{\prime}$ be Cohen-generic over $L$ relative to $G$ such that $p\in G^{\prime}$.
Then $L[G^{\prime}]\models\psi$ but $L[G^{\prime}]\neq L[G]$, a contradiction to the assumption that $\psi$ is a uniqueness statement.
\end{proof}

\begin{remark}
For any $\in$-sentence $\phi(\alpha)$ with ordinal parameters, the homogenity of Cohen forcing implies that, if $\phi(\alpha)$ holds in the generic extension, then $\phi(\check{\alpha})$ is forced by $\mymathbb{1}$. It follows that the sentence $\phi$ just constructed is also $\mathcal{C}$-canonical. 
\end{remark}

This example can be considerably strengthened: In fact, no set forcing extension can be canonical. It is not obvious that the statement `I am not a set forcing extension' is expressable in the first-order
language of set theory at all, but by Reitz \cite{R}, where it is introduced under the name `ground model axiom' or `ground axiom', it turns out to be so.

\begin{defini}
[See \cite{R}] The Ground Model Axiom (GMA) is the statement that there is no transitive class model $M$ of ZFC such that, for some forcing $\mathbb{P}\in M$ and some $\mathbb{P}$-generic filter $G$ over $M$,
we have $V=M[G]$. It is proved in \cite{R} that GMA is expressible as an $\in$-formula.
\end{defini}

\begin{thm}{\label{gmacn}}
 The ground model axiom GMA is canonically necessary.
\end{thm}

\begin{proof}
Assume that $M$ is canonical, witnessed by $\phi$, and $M$ does not satisfy the ground axiom, e.g. $M=N[G]$, where $N$ is an inner model of $M$ and
$G$ is a generic filter for a forcing $\mathbb{P}\in N$. As $\phi$ holds in $M$, there is some $p\in\mathbb{P}$ such that $p\Vdash\phi$ over $N$.

We pass from $M$ to a generic extension $M[H]$ in which $\mathfrak{P}^{M}(\mathbb{P})$ is countable (via the appropriate Levy collapse). 
By Rasiowa-Sikorski, we find in $M[H]$ two mutually $\mathbb{P}$-generic filters over $N$ containing $p$, namely $G_{1},G_{2}$. 
Hence $N[G_{1}]\models\phi$ and $N[G_{2}]\models\phi$, but $N[G_{1}]\neq N[G_{2}]$, as e.g. $G_{1}\in N[G_{1}]\setminus N[G_{2}]$, so $M=N[G]$ cannot be unique with this property, a contradiction.

\end{proof}

Although the preceding argument generalizes to show that forcing extensions by homogenous, as well as countably closed, forcings cannot be $\mathcal{C}$-canonical, we do not know whether GMA is $\mathcal{C}^{L}$-canonical. It is, however, not $\mathcal{C}$-canonical:

\begin{thm}{\label{gma not c canonical}}
GMA is not $\mathcal{C}$-canonically necessary.
\end{thm}
\begin{proof}
(Sketch) We use an iterated forcing of length $\omega$. (A similar construction is described in Reitz \cite{Reitz}, proof of Theorem 26.) The first forcing $\mathbb{P}_{0}$ adds a Cohen-real, and then, for $i>1$, the $i$th forcing encodes the generic objects of the $(i-1)$th forcing in the continuum function. In the extension by this forcing, the generic filter can then be read off from the continuum function. Let $\mathbb{P}$ denote this forcing, and, for some $\mathbb{P}$-generic extension $L[G]$ of $L$, let $T$ be the theory that tells us sufficiently many values of the continuum function in $L[G]$ to reconstruct $G$. Then $L[G]$ is the unique transitive class model of $T$+`I am a $\mathbb{P}$-generic extension of $L$'. Hence $L[G]$ is $\mathcal{C}$-canonical, but $L[G]\not\models$GMA. 
\end{proof}

\begin{question}
Is GMA $\mathcal{C}^{L}$-canonically necessary?
\end{question}

By a similar argument, now using the forcing theorem for symmetric extensions (see Hayut and Karagila, \cite{HK}, p. 453; also see Karagila \cite{Karagila}), we obtain:

\begin{corollary}{\label{ground model axiom for symmetric models}}
No symmetric extension of a ZFC-model is canonical. Thus, the statement that $V$ is not a generic extension of a ZFC-model is canonically necessary. 
\end{corollary}

We can further exploit this argument in a different direction, yielding an upper bound on the elements of a canonical model. To this end, recall, e.g. from Fuchs, Hamkins and Reitz \cite{FHR} that, in set-theoretical geology, the mantle is the intersection of all transitive class models $M$ of ZFC such that $V$ is a generic extension of $M$. 

\begin{prop}{\label{canonical mantle}}
If $M$ is canonical, then $M$ is a subclass of the mantle.
\end{prop}
\begin{proof}
Suppose that $M$ is canonical with uniqueness statement $\phi$ and $N$ is a ground for $V$. Then $V=N[G]$ for some $\PP$-generic filter $G$ over $N$ and some forcing $\PP\in N$. Moreover, for some formula $\psi(x,y)$ and some parameter $y$, there is a condition $p\in G$ which forces over $N$ that $\varphi$ holds in the inner model that is defined by a  $\psi(x,y)$. 
Let $H$ be $\PP$-generic over $V$ with $p\in H$. Then $G$ and $H$ are mutually generic. 
Since $p\in H$, there is an inner model $M^{\prime}$ of $N[H]$ in which $\varphi$ holds. 
By uniqueness of inner models of $\varphi$ in $N[G\times H]$, we have $M=M^{\prime}$. Since $G$ and $H$ are mutually generic over $N$, we have $N[G]\cap N[H]=N$ and hence $M\subseteq N$, as required. 

\end{proof}

Given Theorem \ref{gmacn}, one might wonder whether GMA captures the full strength of canonical necessity, i.e. whether there are canonically necessary statements that do not follow from GMA. This also turns out to be true:

\begin{thm}{\label{nomeasurables}}
\begin{enumerate}
\item The statement that `There is no measurable cardinal' is canonically necessary. 
\item The statement that `There are only finitely many measurable cardinals' is $\mathcal{C}^{1}$-canonically necessary. 
\end{enumerate}
\end{thm}
\begin{proof}
We start with the second statement. 
Suppose for a contradiction that $M$ is 
a canonical model with uniqueness statement $\phi(\alpha)$ ($\alpha\in\text{On}$) and that $M$ contains infinitely many measurable cardinals $(\kappa_{i}:i\in\omega)$. By a theorem of Kunen (see, e.g., Theorem 19.17 of \cite{Kanamori}), there is an ultrapower $\text{Ult}(M,U)$ by a normal  ultrafilter $U$ on some $\kappa_{i}$ that fixes $\alpha$. 
But then, $\text{Ult}(M,U)$ is a transitive class inner model of ZFC+$\phi(a\alpha)$ different from $M$, contradicting the assumption that $M$ was unique with these properties. 

The first statement can be proved in the same way, by just taking the ultrapower with any normal ultrafilter on a measurable cardinal in $M$ (since there is no $\alpha$ to preserve). 

\end{proof}

\begin{thm}
Suppose that there is a measurable cardinal in $V$. Then every canonical model has a 
proper class of 
order indiscernibles.
\end{thm}
\begin{proof} 
Suppose that $M$ is a canonical model with uniqueness statement $\phi$, and suppose that $M$ is defined in $V$ as $\{x:\psi(x,\vec{p}\}$, for some $\vec{p}\in V$. Let $M_{0}:=M$. 
Let $U$ be a normal ultrafilter on $\kappa$, and let $\pi:V\rightarrow\text{Ult}(V,U)$ be the ultrapower embedding. Define $M_1$ as $\{x:\psi(x,\pi(\vec{p})\}^{\text{Ult}(V,U)}$. 
Since $\pi$ is elementary, so is $(\pi\upharpoonright M_0):M_{0}\rightarrow M_{1}$. 
Thus $M_1$ is a transitive class model of ZFC+$\phi$, and thus, $M_{1}=M$. 

Let $\kappa_1=\pi(\kappa)$, so $\kappa_{1}>\kappa$. Iterating this procedure, and taking direct limits at limit ordinals, we obtain a sequence $(M_{\iota}:\iota\in\text{On})$ of models, all of which are equal to $M$, on the one hand and a strictly increasing sequence $(\kappa_{\iota}:\iota\in\text{On})$ of cardinals in $M$ on the other. But now, it is standard to show that $(\kappa_{\iota}:\iota\in\text{On})$ is a sequence of order indiscernibles for $M$.
\end{proof}

\begin{corollary}
 There are canonically necessary statements that do not follow from GMA.
\end{corollary}
\begin{proof}
 By results of J. Reitz (see \cite{R}), the fine structural models for measurable cardinals satisfy GMA. Hence, the nonexistence of measurable cardinals does not follow from GMA.
\end{proof}

The canonical impossibility of measurable cardinals suggests further considerations about the canonical possiblity, or otherwise,
of large cardinals. First, a positive observation:

\begin{thm}{\label{sharp closure}}
(i) Let $\phi^{\sharp}$ be the statement `For every $x\subseteq\omega$, $x^{\sharp}$ exists'.
If $\phi^{\sharp}$ holds in $V$, then $\phi^{\sharp}$ is canonically possible.

(ii) If $\phi^{\dagger}$ is the statement that $x^{\dagger}$ exists for  every $x\subseteq\omega$,
and $\phi^{\dagger}$ holds in $V$, then $\phi^{\dagger}$ is canonically possible.
\end{thm}
\begin{proof}
(i) Let us add the unary function symbol $\sharp$ to the language of set theory, with the obvious intended meaning.
For a set $X$, let $\text{Def}^{\sharp}(X)$ be the set of
subsets of $X$ that are definable over $X$ using the language of set theory extended by $\sharp$. 
Now define a class $\hat{L}\subseteq V$ in analogy with $L$ as follows: 

\begin{itemize}
 \item $\hat{L}_{0}=\emptyset$
 \item $\hat{L}_{\alpha+1}=\text{Def}^{\sharp}(\hat{L}_{\alpha})$
 \item $\hat{L}_{\lambda}=\bigcup_{\iota<\lambda}\hat{L}_{\iota}$ for $\iota<\lambda$
 \item Finally, $\hat{L}:=\bigcup_{\iota\in\text{On}}\hat{L}_{\iota}$.
\end{itemize}

By $\Pi^{1}_{2}$-absoluteness of sharps, the sharp function and thus $\hat{L}$ is a definable subclass of $V$ via the recursion
just stated.
It is clear from the definition that $\hat{L}\models\phi^{\sharp}$. That $\hat{L}\models\text{ZFC}$ can be checked similarly
to the fact that ZFC holds in $L$. 
Thus $\hat{L}$ is a canonical model in which $\phi^{\sharp}$ holds.

(ii) is now proved similarly, as the $\dagger$-operation is still $\Pi^{1}_{2}$.
\end{proof}

We can do the same for other inner model operators whose iterability is $\Pi^1_2$. 

\subsection{Canonicity and the continuum function}

All examples of canonical models that we have constructed so far satisfy the continuum hypothesis, and even the generalized continuum hypothesis. 
An attractive question is then whether ZFC$\models_{c}$CH, or even whether ZFC$\models_{c}$GCH. 
We conjecture that this is false (see the last section). 
Here, we mention some related results. We denote by $c$ the continuum (class) function that maps each ordinal $\alpha$ to $2^{\text{card}(\alpha)}$.

\begin{thm}{\label{nonCHweaklynecessary}}
$\neg CH$ is weakly c.p.
\end{thm}
\begin{proof}
In Theorem 19 of \cite{G}, a set forcing extension $M$ of $L$ is constructed such that $M\models\neg\text{CH}$, but
CH holds in every transitive class $N$ such that $L\subseteq N\subseteq M$ (and this is provable in ZFC). The forcing $\mathbb{P}$ used is definable over $M$ without parameters.

Consider the $\in$-statement $\phi$ `I am a $\mathbb{P}$-extension of $L$'. Then every proper inner model of a transitive class model $M$ of ZFC+$\phi$ will satisfy CH, so that
$M$ is the only inner model of $M$ in which CH fails (and all of this is provable in ZFC+$\phi$). Hence $M$ is a weakly canonical model of ZFC+$\neg$CH.
 
\end{proof}

\subsection{How canonical is ZFC?}

Besides asking which statements are canonically necessary over ZFC, one can also step backwards and ask whether the axioms of ZFC canonically follow from weaker subtheories. 

Even forgetting about uniqueness statements, the restriction to transitive class models alone yields certain results. 

The following observation was communicated to us by Philip Welch:

Recall that $\Sigma_2$-KP is Kripke-Platek set theory (see, e.g., \cite{MW}) supplemented with the $\Sigma_2$-collection and the $\Sigma_2$-separation scheme.
Moreover, Pot denotes the power set axiom.

\begin{thm}{\label{sigma2-kp is enough}}
(Welch) All axioms of ZF are implied over transitive classes by $\Sigma_{2}$-KP+Pot (the power set axiom). In particular, $\Sigma_{2}$-KP+ Pot$\models_{c}$ZF.
\end{thm}
\begin{proof}
Let $M$ be an transitive class model $\Sigma_2$-KP+Pot, defined in $V$. 
By Lemma \ref{ZFmodelchar}, it suffices to show that $M$ is closed under G\"odel functions and almost universal. 

Since $M\models$KP, $M$ is clearly closed under G\"odel functions. 

It remains to show that $M$ is almost universal. 
To this end, we will show how to define inside $M$ the cumulative hierarchy $(V_{\alpha}^{M}:\alpha\in\text{On})$ and show that 
$M=\bigcup_{\iota\in\text{On}}V^{M}_{\iota}$.

Once this is done, the argument finishes as follows: If $x\subseteq M$ is a set, we consider the functional class $F:x\rightarrow\text{On}$ (defined in $V$) that maps each $y\in x$ to the smallest $\iota\in\text{On}$ 
such that $y\in V^{M}_{\iota}$. 
By Replacement in $V$, $F[x]$ is a set of ordinals in $V$, and so $\alpha=\text{sup}F[x]$ is an ordinal. Then, by definition of $\alpha$, we have $x\subseteq V^{M}_{\alpha}$, which suffices.

Inside $M$, define a class function $F:\text{On}\times M\rightarrow M$ as follows:
 
For successor ordinals, let $F(\iota+1,x)=z$ if and only if $z=\mathcal{P}(x(\iota))$, i.e. if and only if there is exactly one set $a$ with $(\iota,a)\in x$ and $\forall{y}(y\in z\leftrightarrow y\subseteq x)$; if there is no such set, then $F(\iota+1,x)=\emptyset$. 
Since the first existential (and implicit, via the uniqueness condition, universal) quantifier is bounded, this is $\Pi_1$.

For a limit ordinal $\delta$, $F(\delta,x)=z$ if and only if $z=\bigcup{x}$, i.e., if and only if $\forall{y}(y\in z\leftrightarrow\exists{b\in x}(y\in b)$. Again, this is $\Pi_1$. 

Now the class function $V^{M}:\text{On}\rightarrow M$ defined by applying the recursion principle to $F$ is $\Sigma_2$ over $M$.

We claim that $V^{M}:=\bigcup_{\iota\in\text{On}}V^{M}_{\iota}=M$. 

It is clear that 
$V^{M}\subseteq M$. 

To see the reverse direction, suppose otherwise, and let $x\in M$ be $\in$-minimal (by transitivity of $M$) such that $x\notin V^{M}$. 
Then, by transitivity of $M$, we have $x\subseteq M$, and by minimality of $x$, we have $x\subseteq V^{M}$. Thus (by replacement in $V$), there must be an ordinal $\alpha$ such that $x\subseteq V^{M}_{\alpha}$. But then, we have $x\in V^{M}_{\alpha+1}$.

\end{proof}

\begin{corollary}
Let $M$ be a canonical model of ZF with uniqueness statement $\phi$, where $\phi$ is $\Sigma_m$, and let $n=\text{max}\{3,n\}$. Then $M$ has no proper transitive $\Sigma_{n}$-elementary substructure that contains all ordinals. 
\end{corollary}
\begin{proof}
Suppose that $N\prec_{\Sigma_{n}}M$ and $\text{On}\subseteq N$. Then $N\models\phi$. Moreover, since $n\geq 3$, $N\models\Sigma_2$-KP and $N\models$Pot. By Theorem \ref{sigma2-kp is enough}, $N\models$ZF. Hence $N\models$ZF+$\phi$, and since $\phi$ is a uniqueness statement, we have $N=M$.
\end{proof}

\begin{remark}
By a similar argument, one can see that a canonical model $M$ with a $\Sigma_n$-uniqueness statement coincides with the $\Sigma_n$-Skolem hull of the ordinals in $M$. 
\end{remark}

\begin{question}
Is there a transitive proper class inner model of KP+Pot in which ZF fails? 
Do we have KP$\models_{c}$ZF?
\end{question}

\subsubsection{The axiom of choice}

We do not know whether there is a canonical model of ZF in which the axiom of choice fails. 
One of the most studied theories contradicting choice is ZF+AD, where AD is the axiom of determinacy. 
Whether or not AD is canonically possible is open. 
However, we have the following partial results:

\begin{thm}{\label{L[R] and AD}}
Assume that $A^{\sharp}$ exists for every $A\subseteq\mathbb{R}$.
Then there is no canonical model of T:=ZF+V=$L(\mathbb{R})$+AD.
\end{thm}
\begin{proof}
Suppose that $M$ is a canonical (proper class) inner model 
and that $T$ holds in $M$; let $\phi$ be the canonical sentence for $M$.

We first claim that $\mathbb{R}^M$ is countable. 
Otherwise, we can find a countable elementary substructure $R^\#$ of $(\mathbb{R}^M)^\#$, but then $L(R)$ is an inner model $M^{\prime}$ that has the the same theory as $M$; in particular, $M\models\phi$.
However, we will clearly have $M^{\prime}\neq M$, contradicting the uniqueness of $M$.

It now follows from Theorem 0.1 in \cite{ScSt} that we can pass to a generic extension of $M$ and find an extender 
sequence $\vec{E}$ with extenders below $\omega_1^M$ such that the symmetric 
collapse of $\mathrm{Col}(\omega,{<}\omega_1^M)$ over $M$ is equal to $L(\mathbb{R})^V$ (in fact, it is shown in \cite{ScSt} that $\omega_1^M$ is a limit of Woodin cardinals in $L[\vec{E}]$). 

Since $\mathbb{R}^M$ is countable and since $(\mathbb{R}^M)^\#$ exists, the required 
generic extension exists in $V$, as it can be chosen as a collapse below the least $\mathbb{R}^M$-indiscernible. 
Thus we can choose $L[\vec{E}]$ in $V$ and since $\omega_1^M$ is countable and $\vec{E}^\#$ 
exists, $P(\omega_1^M)^M$ is countable in $V$. Hence we can choose two different $\mathrm{Col}(\omega,{<}\omega_1^M)$-generics 
over $M$ in $V$. By homogeneity of the forcing, the extensions have the same theory, but this contradicts the uniqueness of $M$. 
\end{proof}

\section{Theory Canonicity}

The fact that canonicity rules out 
measurable cardinals almost trivially may suggest that it is too strong a demand; and indeed, uniquely fixing a transitive class model is rather much to ask.
A reasonable weakening would be that the statement $\phi$ only fixes transitive class models up to elementary equivalence, which is the intuition behind the concept of theory-canonicity, as defined above.

We observe that theory-canonicity still goes beyond ZFC, but allows for measurable cardinals:

\begin{prop}{\label{possible measurables}}
 The existence of measurable cardinals is theory-canonically possible.
\end{prop}
\begin{proof}
 Let $\phi$ be the statement `$V=L[U]$, where $U$ is a normal ultrafilter'. Now, every two transitive class models $M$ and $M^{\prime}$ of $\phi$ can be coiterated to a common target model $N$ (see, e.g., \cite{Je}, Theorem 19.14),
 which will be elementary equivalent to both $M$ and $M^{\prime}$; thus $M$ and $M^{\prime}$ are elementary equivalent. Hence $\phi$ is tc. Clearly, any transitive class model $M$ of $\phi$ 
 contains a measurable cardinal.
\end{proof}

\begin{prop}{\label{tc informative}}
 There is an $\in$-sentence $\psi$ that is theory-canonically necessary, but does not follow from ZFC.
\end{prop}
\begin{proof}
For $x\subseteq\omega$, let $T_{x}$ be the theory $\{2^{\aleph_{i}}=\aleph_{i+1}:i\in x\}\cup\{2^{\aleph_{i}}=\aleph_{i+2}:i\notin x\}$. Thus, in a model of $T_{x}$, $x$ is encoded in the continuum function.
Denote the Easton forcing that forces $T_{x}$ to be true over a model of GCH by $P_{x}$.

Now let $\phi$ be the sentence that claims that $V$ is not of the form $(L[x])[P_{x}]$, where $x$ is Cohen-generic over $L$. Clearly, $\phi$ does not follow from ZFC, as it e.g. fails in models of the form $(L[x])[P_{x}]$ just 
described.

We claim that $\phi$ is tcn. Suppose otherwise; then there are a tc statement $\psi$ and a transitive class model $M$ of ZFC+$\psi+\neg\phi$. Then there is a real number $x$
which is Cohen-generic over $L$ such that $M=(L[x])[P_{x}]$. By the forcing theorem, $\psi$ is forced by some Easton-condition $q$ in $P_{x}$.
The fact that $q$ belongs to $P_{x}$ and that $\psi$ is forced by $q$ in $P_{x}$, which holds in $L[x]$, is in turn forced by some Cohen-condition $p$ over $L$.

Now let $r$ and $r^{\prime}$ be two incompatible extensions of $p$, and let $x_{r}$ and $x_{r^{\prime}}$ be Cohen-generic reals extending $r$ and $r^{\prime}$, respectively.

Then both in $L[x_{r}]$ and in $L_[x_{r^{\prime}}]$, we have that $q\Vdash\psi$. Pick a $P_{x_{r}}$-generic filter $G_{r}$ and a $P_{x_{r^{\prime}}}$-generic filter $G_{r^{\prime}}$, both containing $q$; 
this is possible as $q$ belong both to $P_{x_{r}}$ and to $P_{x_{r^{\prime}}}$, because $r$ and $r^{\prime}$ extend $p$, which forces this.

Then $\psi$ holds both in $M_{r}:=(L[x_{r}])[G_{r}]$ and in $M_{r^{\prime}}:=(L[x_{r^{\prime}}])[G_{r^{\prime}}]$; however, as $x_{r}\neq x_{r^{\prime}}$, there is some $i\in\omega$ such that 
$M_{r}\models 2^{\aleph_{i}}=\aleph_{i+1}$ and $M_{r^{\prime}}\models 2^{\aleph_{i}}=\aleph_{i+2}$ or vice versa. Thus $M_{r}$ and $M_{r^{\prime}}$ are two models of 
ZFC+$\psi$ that are not elementary equivalent, which contradicts the assumption that $\psi$ is tc.
\end{proof}

Concerning the size of the continuum however, theory canonicity has little information to offer:

\begin{prop}{\label{tc continuum}}
 When $\iota$ is definable in $L$ and $\text{cf}(\iota)>\omega$, then there is a tc model of ZFC+$2^{\aleph_{0}}=\aleph_{\iota}$. That is, any possible value of the continuum that is definable is also
 theory-canonically possible.
\end{prop}
\begin{proof}
 Let $P_{\iota}$ be the Cohen forcing notion for achieving $2^{\aleph_{0}}=\aleph_{\iota}$ over $L$. Then $P_{\iota}$ is definable from $\iota$, which is definable by assumption,
 so $P_{\iota}$ is definable in $L$. Now the statement $\phi\equiv$`$V$ is a $P_{\iota}$-generic extension of $L$' in a transitive class model $M$ of 
 ZFC implies that $M\models2^{\aleph_{0}}=\aleph_{\iota}$ and moreover, by homogenity of $P_{\iota}$, all sentences true in $M$ are forced by $1$ and thus hold in all such extensions. 
 Thus $M$ is a tc model in which $2^{\aleph_{0}}=2^{\aleph_{\iota}}$ holds.
\end{proof}

\begin{remark}: Note that the proof of Proposition \ref{tc continuum} also implies that GMA is not theory-canonically necessary, as all Cohen-extensions of $L$ have the same theory, `$V$ is a Cohen-extension of $L$'
is expressable as an $\in$-sentence $\phi$ and GMA clearly fails in every transitive class model of $\phi$.
\end{remark}

We can also show that $\neg\text{AC}$ is theory-canonically possible. 
We know the following lemma from a talk of Karagila; however, it has apparently so far not been made explicit in the literature. We thank Asaf Karagila for a sketch of the proof (personal communication) and the kind permission to include it here. 

We say that a transitive class model of ZF is a `Cohen-model of the first type' if and only if it is a symmetric extension of $L$ in the sense of the first Cohen model as described, e.g., in \cite{JechAC}, chapter 5.3. 

\begin{lemma}{\label{symmext axiomatizable}}
There is an $\in$-formula $\phi_{\text{cs}}$ such that, for a transitive class $M\models$ZF, we have that $M\models\phi_{\text{cs}}$ if and only if $M$ is a Cohen model of the first type.\footnote{We thank Asaf Karagila for informing us about this following folklore result and sketching the proof below to us.} 
\end{lemma}
\begin{proof}
By Theorem 4.2 of 
\cite{KS}, we have that, if $M$ is a Cohen model of the first type, then there is an $M$-generic filter $F$ such that $M[F]=L[F]$, $L[F]\models$ZFC and $M$ is a symmetric submodel of $L[F]$. By Remark 3.3 of \cite{KS}, we have $M=L(A)$, where $A$ is the set of Cohen reals added by $F$ over $L$. 

So we can express $\phi_{\text{cs}}$ as `There is a generic filter $F$ such that $M$ is the symmetric submodel of $L[F]$'. (More precisely, we mean that $M$ is the class of the $F$-evaluations of the symmetric names.)  
\end{proof}

\begin{prop}{\label{tc choice}}
There is a canonical theory $T$ that extends ZF+$\neg\text{AC}$.
\end{prop}
\begin{proof}

Consider the statement $\phi_{\text{cs}}$ from Lemma \ref{symmext axiomatizable}, and let $M$ be a transitive class model of ZFC+$\phi_{\text{cs}}$. 

For such an $M$, we have $M\models\neg$AC; moreover, if $\psi$ is a sentence such that $M\models\psi$, then, by homogenity of the Cohen-forcing, we have $\mymathbb{1}\Vdash\psi$, and hence, any two transitive class models of $\phi_{\text{cs}}$ are elementary equivalent.

\end{proof}

\section{Further Ideas and Questions}

\begin{question}
Is $V=HOD$ (theory-)canonically necessary?
\end{question}

\begin{question}
Is CH canonically necessary? Is GCH canonically necessary? 
\end{question}

We conjecture that the answer to the last question is negative; a natural approach is to iterate Friedman's forcing for adding a $\Pi^{1}_{2}$-singleton (described in \cite{Fr}, chapter $6$) $\omega_2$ many times to generate a canonical model in which CH fails.

\begin{question}
(Dominik Klein): Is AC canonical for models of ZF? That is, is there a formula $\phi$ such that there is a unique transitive class model $M$ of ZF$+\phi$ and such that AC fails in $M$? 

More boldly, is there a canonical model of ZF+AD?

\end{question}

By a theorem of Kunen \cite{Kunen}, AC fails in the model $C^{\omega_2}$, which is constructed in analogy with $L$, but using definability in the infinitary language $L_{\omega_{1},\omega_{1}}$. This motivates the investigation of canonical implication for infinitary languages, which we plan to consider in future work.

\begin{question}
Can $M_1^\#$ be an element of a canonical model? 

\end{question}

\begin{question}
In general, which ZFC axioms are (theory-)canonical over the others? Or over KP? Are there e.g. canonical models for ZFC$^{-}$ in which power set fails? Are there canonical models of ZFC without replacement in which replacement is false?
\end{question}

The Ehrenfeucht principles asserts that, when $a$ is definable from $b$, but $a\neq b$, then $a$ and $b$ have different types; there are several variantes of this principle, discussed in, e.g., \cite{Enayat}, \cite{Mycielski}, \cite{FGH}. 
 Following \cite{Enayat1}, a model $M$ of ZF is `Leibnizean' if and only if any two elements of $M$ have different types with ordinal parameters.
 
 \begin{question}
 Does every canonical model of ZFC satisfy the Ehrenfeucht principle? Is every canonical model of ZFC Leibnizean?
 \end{question}

\begin{question}
Is there a proof calculus that is sound and complete for $\models_{\text{c}}$ and $\models_{\text{tc}}$? (Note that, by failure of compactness, such a calculus cannot consist of a finite set of finite rules.) 
Are there inference rules $\mathcal{R}$ such that, if $\mathcal{R}$ allows deducing $\phi$ from $T$, then $T\models_{c}\phi$ (or $T\models_{\text{tc}}\phi$), but $\phi$ can fail in transitive class models of $T$?
\end{question}

A concept closely related to, but different from, (theory-)canonicity is that of `unique describability': Namely, we say that $\phi(\alpha)$ has \textit{unique models} if and only if, for every transitive class model $N$ of ZFC, there is at most one inner model $M$ of $N$ such that $M\models\phi(\alpha)$. This more liberal notion is closely related to recognizability by Ordinal Turing Machines (see \cite{CSW}) and is considered in ongoing work with Philip Welch. 

\section{Acknowledgements}

We thank Philip Welch for his kind permission to include Theorem \ref{sigma2-kp is enough} in this work. We also thank Asaf Karagila for sketching a proof of Lemma \ref{symmext axiomatizable} to us.

\section{Funding} 
Open Access funding enabled and organized by Projekt DEAL. 
This research was funded in whole or in part by EPSRC grant number EP/V009001/1 of the second-listed author. For the purpose of open access, the authors have applied a `Creative Commons Attribution' (CC BY) public copyright licence to any Author Accepted Manuscript (AAM) version arising from this submission.

\end{document}